\documentclass[12pt]{article}

\usepackage{mathrsfs}

\usepackage{amssymb, amsthm, amsfonts, amsxtra, amsmath}
\usepackage{latexsym}
\usepackage{mathabx}

\theoremstyle{change}

\newtheorem{Theorem}{Theorem}[section]
\newtheorem{Def}[Theorem]{Definition}
\newtheorem{Lem}[Theorem]{Lemma}
\newtheorem{Prop}[Theorem]{Proposition}
\newtheorem{Cor}[Theorem]{Corollary}
\newtheorem{Not}[Theorem]{Notation}

\newtheorem{Def-Prop}[Theorem]{Definition-Proposition}

\date{}

\begin{document}

\hyphenation{Wo-ro-no-wicz}

\title{Equivariant Morita equivalences between Podle\'{s} spheres}
\author{Kenny De Commer\footnote{Supported in part by the ERC Advanced Grant 227458
OACFT ``Operator Algebras and Conformal Field Theory" }\\ \small Dipartimento di Matematica \& CMTP (Center for Mathematics and Theoretical Physics),\\ \small Universit\`{a} degli Studi di Roma Tor Vergata\\
\small Via della Ricerca Scientifica 1, 00133 Roma, Italy\\ \\ \small e-mail: decommer@mat.uniroma2.it}
\maketitle

\newcommand{\acnabla}{\nabla\!\!\!{^\shortmid}}
\newcommand{\undersetmin}[2]{{#1}\underset{\textrm{min}}{\otimes}{#2}}
\newcommand{\otimesud}[2]{\overset{#2}{\underset{#1}{\otimes}}}
\newcommand{\qbin}[2]{\left[ \begin{array}{c} #1 \\ #2 \end{array}\right]_{q^2}}

\newcommand{\qortc}[4]{\,\;_1\varphi_1\left(\begin{array}{c} #1  \\#2 \end{array}\mid #3,#4\right)}
\newcommand{\qortPsi}[4]{\Psi\left(\begin{array}{c} #1  \\#2 \end{array}\mid #3,#4\right)}
\newcommand{\qorta}[5]{\,\;_2\varphi_1\left(\begin{array}{cc} #1 & #2 \\ & \!\!\!\!\!\!\!\!\!\!\!#3 \end{array}\mid #4,#5\right)}
\newcommand{\qortd}[6]{\,\;_2\varphi_2\left(\begin{array}{cc} #1 & #2 \\ #3 & #4 \end{array}\mid #5,#6\right)}
\newcommand{\qortb}[7]{\,\;_3\varphi_2\left(\begin{array}{ccc} #1 & #2 & #3 \\ & \!\!\!\!\!\!\!\!#4 & \!\!\!\!\!\!\!\!#5\end{array}\mid #6,#7\right)}

\newcommand{\otimesmin}{\underset{\textrm{min}}{\otimes}}
\newcommand{\bigback}{\!\!\!\!\!\!\!\!\!\!\!\!\!\!\!\!\!\!\!\!\!\!\!\!}

\abstract{\noindent We show that the family of Podle\'{s} spheres is complete under equivariant Morita equivalence (with respect to the action of quantum $SU(2)$), and determine the associated orbits. We also give explicit formulas for the actions which are equivariantly Morita equivalent with the quantum projective plane. In both cases, the computations are made by examining the localized spectral decomposition of a generalized Casimir element.}\\


\noindent AMS 2010 \emph{Mathematics subject classification}: 17B37, 81R50, 46L08


\section*{Introduction}

\noindent This paper is concerned with $SU_q(2)$, the quantum $SU(2)$ group, at real values $0<q<1$ (\cite{Wor1}). In \cite{Pod2}, the $SU_q(2)$-homogeneous spaces were classified which have the same spectral decomposition as the ordinary action of $SU(2)$ on the 2-sphere (and whose spin 1-part generates the algebra). They form a continuous one-parameter-family $S_{qc}^2$, called \emph{Podle\'{s} spheres}, and are indexed by a number $c \in \lbrack 0,+\infty\rbrack$. In this paper, we give a classification with respect to a weaker equivalence relation, namely \emph{equivariant Morita equivalence}. The notation we follow in the Introduction will be the one of \cite{Pod2}. (In the paper itself we will use a different notational convention which is more convenient for our purposes).

\begin{Theorem}\label{Theo1} Write \[c:\lbrack 0,+\infty\rbrack \rightarrow \lbrack 0,+\infty \rbrack : x\rightarrow (q^{-x}-q^x)^{-2}.\] Then \[S_{qc(x)}^2 \underset{equiv.}{\underset{SU_q(2)\textrm{-Morita}}{\cong}} S_{qc(y)}^2 \qquad \Leftrightarrow \qquad \exists m\in \mathbb{Z} \textrm{ with } y= |x+m|.\] Moreover, any quantum homogeneous space $\mathbb{X}$ of $SU_q(2)$ which is equivariantly Morita equivalent with a Podle\'{s} sphere is itself a Podle\'{s} sphere.
\end{Theorem}

\noindent \emph{Remark:} The moreover-part follows from the results of \cite{Tom1} and the classification in \cite{Pod2}, but we will give an independent proof.\\

\noindent For the \emph{equatorial Podle\'{s} sphere} $S_{q\infty}^2$, there exists an $SU_q(2)$-equivariant $\mathbb{Z}_2$-symmetry, which allows us to form the \emph{quantum projective plane} $\mathbb{R}P_q^2$ as an $SU_q(2)$-homogeneous space (see e.g.~ \cite{Haj1}). The following theorem provides the classification of quantum homogeneous spaces which are $SU_q(2)$-Morita equivalent with $\mathbb{R}P_q^2$.\\

\begin{Theorem}\label{Theo2} For $l \in \frac{1}{2}\mathbb{N}_0$, let $B_l$ be the unital $^*$-algebra generated by elements $X,Z,Y$ and $4l+1$ elements $A_s$, where $s\in \{-2l,-2l+1,\ldots, 2l-1,2l\}$, satisfying the following relations: \[\hspace{-12cm} \left\{\begin{array}{llllll} Y^* &=&X \\ Z^*&=&Z \\ A_s^* &=& (-1)^sA_{-s},\end{array}\right.\]\[\hspace{-11.5cm}\left\{\begin{array}{llllll} XZ &=& q^2ZX \\ A_s Z &=& -q^{-2s} ZA_s,  \end{array}\right.\] \[\hspace{-7.1cm}\left\{\begin{array}{llllll} XA_s &=& -A_{s-1}(1+q^{2s+2l-1}Z) & \textrm{for }s>-2l \\ XA_{-2l} &=& -A_{-2l}X \\ X^*A_s &=& -A_{s+1}(1-q^{2s-2l+1}Z) &\textrm{for }s<2l\\ X^*A_{2l} &=& -A_{2l}X^*\end{array}\right.\] and \[\left\{\begin{array}{l} \left\{ \begin{array}{lll} X^*X = (1-q^{2l-1} Z)(1+q^{-2l-1}Z) \\ XX^* = (1-q^{2l+1}Z) (1+q^{-2l+1}Z)\end{array}\right.\\ \\ \left\{\begin{array}{lllllll} A_{s}A_{s'} &=& (-1)^s X^{-(s+s')}(q^{2s'-2l+1}Z;q^2)_{s+2l}(-q^{-2l+1}Z;q^2)_{s'+2l}& \textrm{for }s+s'\leq 0\\ &=& (-1)^s(X^*)^{s+s'} (q^{2s'-2l+1}Z;q^2)_{2l-s'}(-q^{2s+2s'-2l+1}Z;q^2)_{2l-s}&\textrm{for }s+s'\geq0.\end{array}\right.\end{array}\right.\] In particular, the unital $^*$-algebra $\textrm{Alg}(X,Z,Y)$ generated by $X,Z,Y$ is an isomorphic copy of the Podle\'{s} sphere at parameter $c(2l)$.\\

\noindent Then we can define on $B_l$ an ergodic action of $SU_q(2)$ which agrees with the usual action on the copy $\textrm{Alg}(X,Z,Y)$ of the Podle\'{s} sphere, and such that \[\theta_{2l} = (q^{\frac{1}{2}s(s-1)} \frac{(q^{4l-2s+2};q^2)_{s+2l}^{1/2}}{(q^{2};q^2)_{s+2l}^{1/2}} A_{s})_{s} \;\in B_{l}\otimes \mathbb{C}^{4l+1}\] is a $\pi_{2l}$-eigenvector (where $\pi_{r}$ for $r\in \frac{1}{2}\mathbb{N}$ denotes the spin $r$-representation of $SU_q(2)$). \\

\noindent If we denote by the formal symbol $\mathbb{X}_l$ the quantum homogeneous space associated with the action on $B_l$, then a quantum homogeneous space $\mathbb{X}$ of $SU_q(2)$ is equivariantly Morita equivalent with $\mathbb{R}P_q^2$ iff it is isomorphic to $\mathbb{R}P_q^2$ or one of the $\mathbb{X}_l$.\end{Theorem}

\noindent Classically (i.e.~ for $q=1$), the $B_l$ correspond to the inductions to $SU(2)$ of the actions $\textrm{Ad}(\pi_l)$ of $D_{\infty}^*$, where $D_{\infty}^{*} \subseteq SU(2)$ is the double cover of the infinite dihedral group $S^1\rtimes \mathbb{Z}_2$, and with the $\pi_l$ denoting its $2$-dimensional irreducible representations. Note that the $B_l$ have a natural equivariant $\mathbb{Z}_2$-gradation.\\

\noindent To prove these theorems, we will proceed as follows. In \cite{DeC2}, we introduced a $^*$-algebra $U_q(-,+)$, equipped with a right module $^*$-algebra structure of $U_q(su(2))$, the quantum universal enveloping $^*$-algebra of $su(2)$. We showed that the Podle\'{s} spheres (for $c\neq 0$) can be realized as equivariant sub-quotients of this $^*$-algebra, by evaluation of a certain central and self-adjoint Casimir element. But as $U_q(-,+)$ also has a compatible \emph{co}-module $^*$-algebra structure for $U_q(su(2))$ (namely, a \emph{Yetter-Drinfel'd} structure), one can compose representations of $U_q(-,+)$ with ordinary representations of $U_q(su(2))$, and split these up into irreducibles (a classical method). From applying such a composition to the mentioned Casimir element, the decomposition can easily be deduced by a (trivial) spectral decomposition. From such a procedure we will then be able to prove Theorem \ref{Theo1}. Also Theorem \ref{Theo2} will be proved in a similar fashion.\\

\noindent Let us remark that by \cite{Tom1}, the computation of the orbit under $\mathbb{G}$-equivariant Morita equivalence for an ergodic action $\alpha$ on a unital C$^*$-algebra $B$ can be found by studying the representation theory of $B\rtimes \mathbb{G}$. From this observation, it follows that our work will be directly connected with \cite{Schm1}, where the infinitesimal version $\textrm{Pol}(S_{qc}^2)\rtimes U_q(su(2))$ is studied from a representation theoretic viewpoint. We will at the appropriate places remark where we make contact with \cite{Schm1}, but on the whole our approach is a little different as we tend to work locally.\\

\noindent The contents of this paper are as follows.\\

\noindent After a section containing notational conventions, our \emph{first section} introduces those quantum group concepts we will need in the paper. Nothing in this section is original, but we provide short proofs for certain statements nevertheless. In the \emph{second section}, we prove Theorem \ref{Theo1}, and as a corollary compute the equivariant Picard group for the Podle\'{s} spheres. In the \emph{third section}, we prove Theorem \ref{Theo2}.\\

\section*{Notations}

\noindent In the remainder of the article, $q$ will denote a real number strictly between 0 and 1. We then write \[\lambda = (q-q^{-1})^{-1}.\] We will also use a different parametrization $\tau$ of $\lbrack -\infty,+\infty\rbrack$, namely \[\tau(x) = q^{-x}-q^{x} \qquad \textrm{for }x\in \lbrack -\infty,+\infty\rbrack.\]

\noindent All our vector spaces will be over the ground field $\mathbb{C}$. For $V$ a vector space, we denote $L(V)$ for the space of linear endomorphisms of $V$, and by $V^{\circ}$ the space of linear functionals. If $V$ is endowed with a Hilbert space structure $\mathscr{H}$, we denote $B(\mathscr{H})$ for the $^*$-algebra of bounded operators. When we have a basis $e_i$ of a vector space $V$, parametrized by a set $I$, then $e_i$ is interpreted to be zero if $i\notin I$.\\

\noindent By $\odot$, we will denote the algebraic tensor product of two vector spaces or algebras over $\mathbb{C}$. By $\otimes$, we will denote the tensor product between Hilbert spaces, or the minimal tensor product between C$^*$-algebras.  We will also use the leg notation for tensor products: for example, if we have spaces $V_1,V_2,V_3$, and $X$ an operator in $L(V_{1}\odot V_{3})$, we denote by $X_{13}$ the operator on $V_1\odot V_2\odot V_3$ acting as $X$ on the first and third component, and as the identity on the second component. \\

\noindent For $r\in \mathbb{N}\cup\{\infty\}$ and $a\in \mathbb{C}$, we denote by $(a;q)_r$ the $q$-factorial \[(a;q)_{r} = \overset{r-1}{\underset{k=0}{\prod}} (1-q^ka).\]

\section{Preliminaries}

\subsection{Quantum groups}\label{SecQG}

\noindent We will freely use the language of Hopf algebras, Hopf $^*$-algebras, and C$^*$-algebraic compact quantum groups (see e.g.~ \cite{Kli1}). For a Hopf algebra $(H,\Delta)$, we will use Sweedler notation in the form \[\Delta(h) = h_{(1)}\otimes h_{(2)}\qquad \textrm{for }h\in H.\] A C$^*$-algebraic compact quantum group will always be written in the form $(C(\mathbb{G}),\Delta)$, and we then refer to the symbol $\mathbb{G}$ as `the compact quantum group'. The associated Hopf $^*$-algebra is written $\textrm{Pol}(\mathbb{G})$. Except for the preliminary section, we will only be interested in these objects for one particular quantum group, namely $\mathbb{G}=SU_q(2)$.  \\

\noindent \emph{Important remark:} As to avoid overloading certain statements, we will in the remainder of this section always assume that \emph{$\mathbb{G}$ is co-amenable}, so that $C(\mathbb{G})$ is uniquely determined by $\textrm{Pol}(\mathbb{G})$.\\

\noindent The following easy lemma will be needed at a certain point. Let $H$ be an algebra, and $V$ a right $H$-module. We then denote by $V_{\textrm{fin}}\subseteq V$ the submodule of all \emph{locally finite elements}, i.e.~ \[V_{\textrm{fin}} = \{v\in V\mid \{v\cdot h\mid h\in H\} \textrm{ is finite-dimensional}\}.\]

\begin{Lem}\label{LemFin} Let $(H,\Delta)$ be a Hopf algebra, and let $V$ and $W$ be two right $H$-modules. Then \[(V\odot W)_{\textrm{fin}} = V_{\textrm{fin}}\odot W_{\textrm{fin}}.\]\end{Lem}

\noindent We also make the following remark. Let $(H,\Delta)$ be a Hopf ($^*$-)algebra, and let $A$ be a right module ($^*$-)algebra for $(H,\Delta)$. (The compatibility with the $^*$-structure means that $(a\cdot h)^* = a^*\cdot S(h)^*$). Let $V$ be a finite-dimensional vector space (resp. Hilbert space) with a left $H$-module structure by a ($^*$-preserving) unital homomorphism $\pi:H\rightarrow L(V)$. Then $A\odot L(V)$ can be made into a right module ($^*$-)algebra by the formula \[(a\otimes x)\cdot h := (a\cdot h_{(2)}) \otimes \pi(S(h_{(1)}))x\pi(h_{(3)}).\] If we are in the following situation: \begin{itemize} \item $(K,\Delta)$ is Hopf ($^*$-)algebra paired with $(H,\Delta)$ by a map $\iota_H: H\rightarrow K^{\circ}$ (with the compatibility \[\iota_H(h^*)(k) = \overline{\iota_H(h)(S(k)^*)}\] in the $^*$-case), \item if the module ($^*$-)algebra structure on $A$ is induced from a left comodule ($^*$-)algebra structure of $K$ on $A$, and \item if $\pi$ is induced from a (unitary) corepresentation $U\in K\odot L(V)$, \end{itemize} then \[(a\otimes x)\cdot h = (\iota_H(h)\otimes \iota\otimes \iota)(U_{13}^{-1}(\alpha(a)\otimes x)U_{13}).\] Also, in the general case, the module $A\odot L(V)$ is isomorphic to $V^{\circ}\odot A\odot V$ with the tensor module structure (where $V$ now carries the right $H$-module structure $v\cdot h := \pi(S(h))v$, and with $V^{\circ}$ endowed with the right module structure $\omega \cdot h := \omega(\pi(h)\,\cdot\,)$).

\subsection{Coactions}

\noindent We begin with the following remark on terminology. We will use the equivalent notions of \emph{(co)module algebra} and \emph{(co)action}, whenever one of them is more convenient. In the C$^*$-algebra context, we will always assume that the co-unit condition is satisfied, so that the coactions are continuous.\\

\noindent Our next remarks concern \emph{ergodic} coactions. We call a coaction $\alpha$ on a unital algebra $B$ \emph{ergodic} if the identity $\alpha(b)=b\otimes 1$ for some $b\in B$ implies that $b\in \mathbb{C}1$. If $\alpha$ is an ergodic coaction of a C$^*$-algebraic compact quantum group $(C(\mathbb{G}),\Delta)$ on a unital C$^*$-algebra $B$, we will write $B = C(\mathbb{X})$ for some formal symbol $\mathbb{X}$, and call it a `$\mathbb{G}$-homogeneous space'. We then denote by $\textrm{Pol}(\mathbb{X})$ the linear span of the finite-dimensional spectral subspaces of $C(\mathbb{X})$. It is a $^*$-algebra carrying a natural coaction of $\textrm{Pol}(\mathbb{G})$ by restricting $\alpha$. One also has a (unique) invariant (and faithful) state $\varphi_{\mathbb{X}}$ on $C(\mathbb{X})$, obtained by integrating out the coaction (so $\varphi_{\mathbb{X}}(x)1_{C(\mathbb{X})} = (\iota\otimes \varphi_{\mathbb{G}})\alpha(x)$ for all $x\in C(\mathbb{X})$, where $\varphi_{\mathbb{G}}$ is the invariant state on $C(\mathbb{G})$). Note that $C(\mathbb{X})$ is completely determined by $\textrm{Pol}(\mathbb{X})$, by our co-amenability assumption on $\mathbb{G}$ (see \cite{Li1}, Proposition 3.8).\\

\noindent The following result by F. Boca (\cite{Boc1}) is fundamental.

\begin{Theorem}\label{TheoFin} Let $\mathbb{X}$ be a homogeneous space for a compact quantum group $\mathbb{G}$. Then any irreducible representation of $\mathbb{G}$ appears in $C(\mathbb{X})$ with only finite multiplicity.
\end{Theorem}

\noindent The following lemma will also be used at some point.

\begin{Lem}\label{LemvN} Let $\mathbb{G}$ be a compact quantum group, $\mathscr{H}$ a Hilbert space, and let $B\subseteq B(\mathscr{H})$ be a (not necessarily closed) unital sub-$^*$-algebra with a coaction $\alpha_B$ by $\textrm{Pol}(\mathbb{G})$. Assume that there exists a normal state $\omega$ in $B(\mathscr{H})_*$ whose restriction to $B$ is faithful and $\alpha_B$-invariant. Then if $A\subseteq B$ is a unital sub-$^*$-algebra for which \begin{itemize}
\item $\alpha_B$ restricts to an \emph{ergodic} coaction of $\textrm{Pol}(\mathbb{G})$ on $A$, and
\item the weak closures of $A$ and $B$ coincide,\end{itemize}
then $A=B$.\end{Lem}

\begin{proof} Suppose that $B\neq A$. We may then take an irreducible representation $\pi$ of $\mathbb{G}$ and a non-zero element $x\in B_{\pi}$, the spectral subspace for $\pi$ in $B$, such that $x\notin A$. As $A_{\pi}$ is finite-dimensional by Boca's theorem, we may moreover assume that $x$ is orthogonal to $A_{\pi}$, and hence to $A$ (where $A$ is equipped with the pre-Hilbert space structure $\langle a',a\rangle := \omega(a^*a')$). But as $\omega$ is normal, we would then get $\omega(xy)=0$ for all $y \in A'' = B''$. Clearly this gives a contradiction with the faithfulness of $\omega$.
\end{proof}

\subsection{Morita equivalence for coactions}\label{SubMor}

\noindent Let $\alpha_i$ be left coactions of $C(\mathbb{G})$ on unital C$^*$-algebras $B_i$. One says the $B_i$ are \emph{$\mathbb{G}$-Morita equivalent} if there exists a unital C$^*$-algebra $E$ with a left coaction $\alpha$, together with a $\mathbb{G}$-invariant self-adjoint projection $e$, such that, denoting $e_1 = e$ and $e_2=1-e$, we have that $Ee_1E$ and $Ee_2E$ are norm-dense in $E$, and $e_iEe_i\cong B_i$ by a $\mathbb{G}$-covariant isomorphism. Alternatively, it is more common to define the $B_i$ to be $\mathbb{G}$-Morita equivalent if there exists an equivariant $B_1$-$B_2$-equivalence Hilbert bimodule (see e.g.~ the remark after Theorem 2.5 in \cite{Nes1}). The equivalence of the latter definition with the above `linking algebra' picture is well-known and easily proven. It is also easily shown that $\mathbb{G}$-Morita equivalence is indeed an equivalence relation. \\

\noindent If the $\alpha_i$ are ergodic, and we write $B_i = C(\mathbb{X}_i)$, we will also call the $\mathbb{X}_i$ themselves $\mathbb{G}$-Morita equivalent.\\

\noindent The following results can be deduced from the ones in section 4 of \cite{Tom1}.

\begin{Prop}\label{LemMor} Let the $\mathbb{X}_i$ be two $\mathbb{G}$-homogeneous quantum spaces. The following are equivalent.
 \begin{itemize} \item The $\mathbb{X}_i$ are $\mathbb{G}$-Morita equivalent.
  \item There exists a finite-dimensional unitary corepresentation $U$ of $C(\mathbb{G})$ on a Hilbert space $\mathscr{H}$ and a $\mathbb{G}$-invariant projection $p\in C(\mathbb{X}_2)\otimes B(\mathscr{H})$ such that \[C(\mathbb{X}_1) \cong p(C(\mathbb{X}_2)\otimes B(\mathscr{H}))p\] by a $\mathbb{G}$-equivariant isomorphism.
  \end{itemize}
\end{Prop}

\noindent Here $C(\mathbb{X}_2)\otimes B(\mathscr{H})$ is again equipped with the coaction $x\rightarrow U_{13}^*(\alpha\otimes \iota)(x)U_{13}$.\\

\noindent To prove $\Rightarrow$, take a $\mathbb{G}$-equivariant equivalence Hilbert bimodule $(\mathcal{E},\alpha_{\mathcal{E}})$ between $C(\mathbb{X}_1)$ and $C(\mathbb{X}_2)$, a suitable unitary left corepresentation $U$ of $C(\mathbb{G})$ and non-zero elements $x_i\in \mathcal{E}$ such that $\alpha_{\mathcal{E}}(x_i) = \sum_j U_{ij}^*\otimes x_j$. Then using the ergodicity of $C(\mathbb{X}_1)$, one shows that (possibly up to a scalar) the map \[\mathcal{E} \rightarrow C(\mathbb{X}_2)\otimes \mathscr{H}_U: \xi \rightarrow \sum_i \langle\xi,\xi_i\rangle_{C(\mathbb{X}_2)} \otimes e_i\] is a $\mathbb{G}$-equivariant isometry between $C(\mathbb{X}_2)$-Hilbert modules, where the range is equipped with the coaction $x\rightarrow U_{13}^*(\alpha\otimes \iota)(x)$. To prove $\Leftarrow$, the essential point is that for any $\mathbb{G}$-invariant projection $p$, the Hilbert module $p (C(\mathbb{X}_2)\otimes \mathscr{H})$ is still full (cf. \cite{Tom1}, Lemma 4.5). This will follow from the fact that $(\iota\otimes \omega)(p) \in C(\mathbb{G})$ is invariant for a well-chosen faithful state $\omega \in B(\mathscr{H})_*$ (namely an invariant functional for the action $x\rightarrow U(1\otimes x)U^*$ by $(\textrm{Pol}(\mathbb{G}),\Delta^{\textrm{op}})$).\\

\noindent The following lemma will allow us to determine Morita equivalences by an inductive process.

\begin{Lem}\label{LemInd} Let $\pi_1,\ldots,\pi_n$ be a generating set of irreducible representations of a compact quantum group $\mathbb{G}$ (i.e.~ any irreducible representation of $\mathbb{G}$ is contained in some power of $\oplus \pi_i$). Let $\mathbb{X}_1$ and $\mathbb{X}_2$ be two $\mathbb{G}$-homogeneous spaces. Then $\mathbb{X}_1$ and $\mathbb{X}_2$ are Morita equivalent iff there exists a finite set of $\mathbb{G}$-homogeneous spaces $\mathbb{Y}_1,\ldots \mathbb{Y}_m$ with \begin{itemize}
\item $\mathbb{Y}_1 \cong \mathbb{X}_1$ and $\mathbb{Y}_m \cong \mathbb{X}_2$,
\item for each $k\in \{1,2,\ldots,m-1\}$, there exists an $i\in \{1,2,\ldots n\}$ and a minimal $\mathbb{G}$-invariant projection $p\in C(\mathbb{Y}_k)\otimes B(\mathscr{H})$ such that \[C(\mathbb{Y}_{k+1}) \cong p(C(\mathbb{Y}_k)\otimes B(\mathscr{H}_{\pi_i}))p.\]
    \end{itemize}
\end{Lem}

\noindent The proof is based on the previous proposition and two basic observations: \begin{itemize} \item If \[C(\mathbb{Y}_1) = p_1(C(\mathbb{Y}_2) \otimes B(\mathscr{H}_{\pi_1}))p_1\quad \textrm{and}\quad C(\mathbb{Y}_2)= p_2(C(\mathbb{Y}_3)\otimes B(\mathscr{H}_{\pi_2}))p_2,\] then with $p_3 = p_1(p_2\otimes 1) = (p_2\otimes 1)p_1$ we have \[C(\mathbb{Y}_1) = p_3(C(\mathbb{Y}_3) \otimes B(\mathscr{H}_{\pi_2\otimes \pi_1}))p_3,\] and
\item If $\pi_1\subseteq \pi_2$ with corresponding projection $p: \mathscr{H}_{\pi_2}\rightarrow \mathscr{H}_{\pi_1}$, then \[C(\mathbb{Y}) \otimes B(\mathscr{H}_{\pi_1}) = (1\otimes p)(C(\mathbb{Y}) \otimes B(\mathscr{H}_{\pi_2}))(1\otimes p).\]
\end{itemize}

\noindent Note that the above two results also (and more naturally) apply to the associated irreducible equivariant $C(\mathbb{X})$-Hilbert modules, i.e.~ any irreducible equivariant Hilbert $C(\mathbb{X})$-module appears as a component in some $C(\mathbb{X})\otimes \mathscr{H}_{\pi}$ for $\pi$ a finite-dimensional representation.\\

\begin{Prop} Let $\mathbb{G}$ be a compact quantum group, and $\mathbb{H}$ a quantum subgroup. Then we can form the $\mathbb{G}$-homogeneous quantum space $\mathbb{X}=\mathbb{H}\backslash \mathbb{G}$. Any $\mathbb{G}$-Morita equivalent homogeneous quantum space is then obtained by taking an irreducible unitary corepresentation $U$ of $C(\mathbb{H})$ on a Hilbert space $\mathscr{H}$, and inducing the associated $\mathbb{H}$-action on $B(\mathscr{H})$ to $\mathbb{G}$.\end{Prop}

\noindent One can use for example the isomorphisms $K_0^{\mathbb{G}}(C(\mathbb{H}\backslash \mathbb{G})) \cong K_0(C(\mathbb{H}\backslash \mathbb{G})\rtimes \mathbb{G}) \cong K_0(C^*(\mathbb{H}))$.\\

\subsection{Galois objects}

\begin{Def}[\cite{Sch1}] Let $(H,\Delta)$ be a Hopf ($^*$-)algebra, $A$ a unital ($^*$-)algebra, and $\alpha$ a right coaction of $(H,\Delta)$ on $A$. Denote $B=\{a\in A \mid \alpha(a)=a\otimes 1\}$, the fixed point algebra. One says $\alpha$ is Galois if the \emph{Galois map} \[G:A\underset{B}{\odot} A\rightarrow A\odot H: a\otimes a'\rightarrow (a\otimes 1)\alpha(a')\] is bijective.\\

\noindent One says $(A,\alpha)$ is a \emph{Galois object} if $\alpha$ is ergodic (i.e. $B= \mathbb{C}$).\end{Def}

\noindent For a Galois object, we write $S$ for the canonical anti-isomorphism $A^{\textrm{op}}\rightarrow A: a^{\textrm{op}}\rightarrow a$, and denote \[ h_{\lbrack 1\rbrack}\otimes h_{\lbrack 2\rbrack} := (S^{-1}\otimes \iota)(G^{-1}(1\otimes h))  \in A^{\textrm{op}}\odot A.\] The application $h\rightarrow h_{\lbrack 1\rbrack}\otimes h_{\lbrack 2\rbrack}$ is then a unital homomorphism. As for $H$ itself, one can make $A$ into a right $H$-module ($^*$-)algebra by means of the \emph{Miyashita-Ulbrich (or adjoint) action} \[ a\lhd h := S(h_{[1]})ah_{[2]}.\] Then $(A,\alpha,\lhd)$ is a right Yetter-Drinfel'd module (see \cite{Doi1}, or Lemma 2.9 of \cite{Sch1}).\\

\noindent A trivial example of a Galois object is given by (an isomorphic copy of) the Hopf ($^*$-)algebra itself, with the coaction given by the comultiplication. In fact, we will be mainly concerned with a particular Galois object for a Hopf $^*$-algebra \emph{which becomes trivial when forgetting the $^*$-structure}.\\

\noindent We record the following fact for later use.

\begin{Lem}\label{LemMod} Let $(A,\alpha)$ be a Galois object for the Hopf $^*$-algebra $(H,\Delta)$. \begin{enumerate}\item Let $\pi_B:A\rightarrow B$ be a unital $^*$-homomorphism. Then there exists a right $H$-module $^*$-algebra structure on $B$, determined by $b\lhd h = \pi_B(S(h_{[1]}))b\pi_B(h_{[2]})$.
\item Let $B$ be as above, and let moreover $(V,\pi)$ be a finite-dimensional left $H$-module. Then \[(\pi_B\otimes \pi)\alpha: A\rightarrow B\otimes L(V)\] is a morphism between right $H$-module $^*$-algebras. Moreover, the module $^*$-algebra structure on $B\otimes L(V)$ (as at the end of Section \ref{SecQG}) coincides with the one induced by this $^*$-homomorphism as in the first point.
\end{enumerate}
\end{Lem}
\begin{proof} The first fact can be proven as in the Hopf $^*$-algebra case (see e.g.~ \cite{Kli1}, Lemma 5.5). For the first part of the second fact, use that $(A,\alpha,\lhd)$ is a right Yetter-Drinfel'd module. For the second part, the following identity for $h\in H$ will imply the claim: \[S(h_{(2)\lbrack 1\rbrack})\otimes h_{(2)\lbrack 2\rbrack} \otimes S(h_{(1)})\otimes h_{(3)} = S(h_{\lbrack 1\rbrack})_{(0)} \otimes h_{\lbrack 2\rbrack (0)} \otimes S(h_{\lbrack 1\rbrack})_{(1)}\otimes h_{\lbrack 2\rbrack (1)}.\] To prove this formula, apply e.g.~ the identities (2.1.3) and (2.1.2) from Lemma 2.1.7 of \cite{Sch1} to the right hand side.
\end{proof}

\noindent Also the following result which will be needed at some point, although only in a very simple situation.

\begin{Prop}\label{PropMor2} Let $\mathbb{G}$ be a compact quantum group, $B$ a unital C$^*$-algebra equipped with an action by $\mathbb{G}$, and $H$ a finite group (or even quantum group) which has a $\mathbb{G}$-equivariant Galois action on $B$. Then $B^{H}$ is $\mathbb{G}$-equivariantly Morita equivalent with $H\ltimes B$.\end{Prop}

\noindent Indeed, by a well-known theorem concerning Galois extensions (\cite{Kre1}), we have that $H\ltimes B \cong \textrm{End}_{B^{H}}(B)$ by the natural homomorphism (where $B$ is considered just as a right $B^H$-module on the right hand side). Clearly this identification is compatible with the $^*$-structure and the $\mathbb{G}$-action, by assumption, leading to the stated equivariant Morita equivalence.\\

\noindent Restating the proposition in the form we will need it in, the above says that, under the given conditions, $p := \frac{1}{|H|}\sum_{h\in H} \lambda_h \in H\ltimes B$ will be a full projection. If moreover $H$ is abelian, and $\chi$ a character, then of course also $p_{\chi} = \frac{1}{|H|}\sum_{h\in H} \chi(h)\lambda_h \in H\ltimes B$ is full, with $p(H\ltimes B)p \cong p_{\chi}(H\ltimes B)p_{\chi}\cong B^{H}$ equivariantly.

\subsection{Quantized universal enveloping algebras}\label{QUEA}

\begin{Def} We denote by $U_q(su(2))$ the quantized universal enveloping $^*$-algebra of $su(2)$. It is the unital algebra generated by elements $E,F,K,K^{-1}$, with commutation relations $KE = q^2EK$, $KF = q^{-2}FK$, $KK^{-1}=1=K^{-1}K$ and \[\lbrack E,F\rbrack = \frac{K-K^{-1}}{q-q^{-1}}.\] The $^*$-operation is determined by $E^*=K^{-1}F$ and $K^* = K$.\\

\noindent We can equip $U_q(su(2))$ with the unital $^*$-homomorphism \[\Delta: U_q(su(2))\rightarrow U_q(su(2))\odot U_q(su(2)),\] uniquely determined by the fact that $\Delta(K)=K\otimes K$ and \begin{eqnarray*} \Delta(E) &=& E\otimes 1 + K^{-1}\otimes E,\\ \Delta(F)&=&F\otimes K + 1\otimes F.\end{eqnarray*} The couple $(U_q(su(2)),\Delta)$ then forms a Hopf $^*$-algebra.\end{Def}

\begin{Def} We denote by $U_q(-,+)$ the $^*$-algebra which, as an algebra, is generated by elements $X,Y,Z,Z^{-1},T$ with commutation relations $XZ = q^{2}ZX$, $YZ = q^{-2}ZY$, $Z^{-1}Z= 1=ZZ^{-1}$ and \[\left\{\begin{array}{lll} YX &=& 1+ q^{-1} TZ - q^{-2}Z^2 \\ XY &=& 1+q\quad TZ-q^{2}\quad \!\!\!\, Z^2\end{array}\right.\] The $^*$-structure is uniquely determined by the formulas $X^* = Y$, $Z^* = Z$.
\end{Def}

\noindent Note that $T$ can be expressed in terms of $X,Y$ and $Z$. Then $T$ can be shown to be central and self-adjoint. It is interpreted as the \emph{Casimir element} of $U_q(-,+)$.\\

\noindent \emph{Remark:} It is easily shown that $U_q(-,+)$ coincides (at least after introducing a square root of $K$) with the $^*$-algebra $\hat{\mathcal{Y}}_c$ from \cite{Schm1}, section 5 for $c\neq 0$ (the $c$ can then be removed by rescaling the parameters). It is also the same $^*$-algebra (again after adjoining a square root of $K$) as the one denoted by the corresponding symbol in \cite{DeC2} (and where the notation is explained), but we take a different presentation now.\\

\begin{Prop} The $^*$-algebra $U_q(-,+)$ can be made into a Galois object for $(U_q(su(2)),\Delta)$ by the coaction $\alpha$, defined on the generators $X,Y,Z,T$ by \begin{eqnarray*} \alpha(Z)&=& Z\otimes K^{-1} \\  \alpha(X) &=& X\otimes 1 +  Z\otimes (q^{-1/2} \lambda^{-1} E),\\ \alpha(Y) &=& Y\otimes 1 + Z\otimes (q^{-1/2} \lambda^{-1}K^{-1}F) \\ \alpha(T) &=& T\otimes K + Z\otimes (\lambda^{-2}FE-q^{-1}(K-K^{-1})) \\&& \qquad + X\otimes (q^{1/2}\lambda^{-1}F)+Y\otimes (q^{1/2}\lambda^{-1}EK). \end{eqnarray*}
\end{Prop}

\noindent This fact can be shown as follows: if we forget the $^*$-structure, then $U_q(-,+)$ is an isomorphic copy of $U_q(su(2))$ by the following identifications: \begin{eqnarray*} X &\leftrightarrow& iq^{-\frac{1}{2}}\lambda^{-1}E \\ Y &\leftrightarrow& iq^{-\frac{1}{2}}\lambda^{-1}K^{-1}F \\ Z &\leftrightarrow& iK^{-1} \\ T &\leftrightarrow& i(\lambda^{-2}EF +qK^{-1}+q^{-1}K).\end{eqnarray*} The above coaction is then in fact just the comultiplication of $U_q(su(2))$, which shows the Galois map is an isomorphism. One should then check the compatibility with the $^*$-operation separately, but this is clear on sight. \emph{Remark that we have thus made a Wick rotation for one Borel subalgebra (generated by $E$ and $K$), but left the remaining part (generated by $F$) unaltered}. This will explain why we will get unilaterally infinite-dimensional representations of our $^*$-algebra later on. Also note that via the above isomorphism, $T$ is identified with an imaginary scalar multiple of the Casimir element of $U_q(su(2))$, but for the new $^*$-structure it is self-adjoint.\\

\noindent Let us remark that the way in which $\hat{\mathcal{Y}}_c$ appears in \cite{Schm1} might lead one to think it could be a Galois object. For $\hat{\mathcal{Y}}_c$ can be seen as the relative commutant (or centralizer) of $\textrm{Pol}(S_{qc}^2)$ inside $\textrm{Pol}(S_{qc}^2)\rtimes U_q(su(2))$. It is easy to check that a dual coaction on a smash (or crossed) product always restricts to the centralizer of the copy of the original algebra, so that we deduce from the above that $\hat{\mathcal{Y}}_c$ will indeed by a right $U_q(su(2))$-comodule $^*$-algebra. Now the dual coaction on a smash product is always a Galois coaction. So one might naively believe that the restriction to the centralizer will then also be Galois, but this is not true in general. However, in the present case $\hat{\mathcal{Y}}_c$ splits of as a tensor product (\cite{Fio1}), and by this fortuitous instance the restricted coaction \emph{does} become a Galois object. To illustrate the subtleness of this situation, we mention that the associated analytic result is \emph{not} true: the relative commutant of $\mathscr{L}^{\infty}(S_{qc}^2)$ inside $\mathscr{L}^{\infty}(S_{qc}^2)\rtimes SU_q(2)$ \emph{does not} become a Galois object (or even a Galois action) for $\mathscr{L}(SU_q(2))$ (the analytic version of $U_q(su(2))$). However, one can remedy this situation in another way, and we will come back to this in future work.\\

\noindent As $U_q(-,+) \cong U_q(su(2))$ when the $^*$-structure is ignored, we can deduce the following result from \cite{Jos1}.

\begin{Prop} The locally finite elements of $U_q(-,+)$ w.r.t.~ the adjoint action form a unital $^*$-algebra $U_q^{\textrm{fin}}(-,+)$, generated by $X,Y$ and $Z$.\end{Prop}

\noindent One can also easily represent $U_q^{\textrm{fin}}(-,+)$ by generators and relations in a similar way as $U_q(-,+)$.\\

\noindent Let us present the concrete formulas for the adjoint action on $U_q(-,+)$ by $U_q(su(2))$. If $b\in U_q(-,+)$, then \begin{eqnarray*} b \lhd K &=& ZbZ^{-1} \\ b\lhd (q^{-1/2}\lambda^{-1}E) &=&  Z^{-1}\lbrack b,X\rbrack\\ b\lhd (q^{3/2} \lambda^{-1} F) &=&  \lbrack b,Y\rbrack Z^{-1}.\end{eqnarray*}

\noindent From the foregoing, we immediately see that there exists a $U_q(su(2))$-equivariant $^*$-automorphism $\sigma$ of $U_q(-,+)$, determined by \[\sigma: U_q(-,+)\rightarrow U_q(-,+): b \rightarrow -b \qquad \textrm{for }b\in\{X,Z,Y,T\}.\]

\subsection{The compact quantum group $SU_q(2)$}

\noindent We will not need to know the explicit form of $C(SU_q(2))$ or $\textrm{Pol}(SU_q(2))$, and therefore simply recall from e.g.~ \cite{Kli1}, section 4.4 that there exists a non-degenerate pairing between $\textrm{Pol}(SU_q(2))$ and $U_q(su(2))$. One then has the following result.

\begin{Prop}\label{PropEquiv} There is a one-to-one-correspondence between the following two structures:
 \begin{itemize}\item Left coactions of $(C(SU_q(2)),\Delta)$ with a finite-dimensional space of invariant elements.
 \item Right module $^*$-algebras $A$ for $U_q(su(2))$ such that
 \begin{itemize}\item[$\bullet$]  $A_{\textrm{fin}} = A$.
 \item[$\bullet$]  All eigenvalues for the action of $K$ are positive.
 \item[$\bullet$] The space of $a\in A$ with $a\cdot g = \varepsilon(g)a$ for all $g\in U_q(su(2))$ is finite-dimensional.
 \item[$\bullet$] There exists a faithful unital $^*$-homomorphism of $A$ into a unital C$^*$-algebra.
 \end{itemize}
 \end{itemize}
 \end{Prop}

\subsection{Podle\'{s} spheres}

\noindent \emph{Warning:} For notational reasons, we will follow a slightly different convention than the more common one used in the Introduction: we will use the index $\tau(x) = q^{-x}-q^{x}$ instead of $c(x) = \tau(x)^{-2}$.

\begin{Def}\label{DefPod} Let $x \in (-\infty,+\infty)$, and denote $\tau = \tau(x)$. The $^*$-algebra $\textrm{Pol}(S_{q\tau}^2)$ is generated by three elements $X_{\tau},Z_{\tau},Y_{\tau}$ with $X_{\tau}^*=Y_{\tau}$, $Z_{\tau}^*=Z_{\tau}$, $X_{\tau}Z_{\tau}=q^{2}Z_{\tau}X_{\tau}$ and with \[\left\{\begin{array}{lll} X_{\tau}^*X_{\tau} &=& (1-q^{x-1}Z_{\tau})(1+q^{-x-1}Z_{\tau})\\ X_{\tau}X_{\tau}^* &=& (1-q^{x+1}Z_{\tau})(1+q^{-x+1}Z_{\tau}).\end{array}\right.\] It carries (up to isomorphism) a unique right $U_q(su(2))$-module $^*$-algebra structure, induced from a left $\textrm{Pol}(SU_q(2))$-coaction, for which the span of the $1,X_{\tau},Z_{\tau},X_{\tau}^*$ is a direct sum of the trivial and the spin 1-representation of $SU_q(2)$. The corresponding action of $SU_q(2)$ is then ergodic.\\

\noindent We call the symbol $S_{q\tau}^2$ the \emph{Podle\'{s} sphere at parameter $\tau$}. When $\tau=0$, we call it the \emph{equatorial Podle\'{s} sphere}.
\end{Def}

\noindent \emph{Remarks:} \begin{enumerate} \item One also has the \emph{standard Podle\'{s} sphere} $S_{q\infty}^2\cong S^1\backslash SU_q(2)$. As it is degenerate from our point of view, we will treat it separately later on.
\item There is an equivariant $^*$-isomorphism $\sigma_{\tau}$ from $\textrm{Pol}(S_{q\tau}^2)$ to $\textrm{Pol}(S_{q,-\tau}^2)$ sending $b_{\tau}$ to $-b_{-\tau}$ for $b\in \{X,Z,Y,T\}$. Hence up to isomorphism, $\textrm{Pol}(S_{q,\tau}^2)$ only depends on $|\tau|$, and we can parametrize Podle\'{s} spheres by $c = \frac{1}{\tau^2}$. The latter is the convention we used in the Introduction. For the purposes of the article, it will be more convenient \emph{not} to identify the two Podle\'{s} spheres immediately. For example, on the equatorial Podle\'{s} sphere we get in particular an involutive equivariant automorphism $\sigma_0$, which plays an important r\^{o}le in the theory.
\end{enumerate}

\noindent The following was proven in \cite{DeC2}, but can be immediately verified. We will denote by $\textrm{Pol}^{\textrm{ext}}(S_{q\tau}^2)$ the $^*$-algebra which is obtained by adjoining to $\textrm{Pol}(S_{q\tau}^2)$ an inverse of $Z$ (which clearly does not introduce additional relations).

\begin{Prop}\label{PropQuot} There is a $U_q(su(2))$-equivariant unital $^*$-homomorphism \[\pi_{\tau}: U_q(-,+)\rightarrow \textrm{Pol}^{\,\textrm{ext}}(S_{q\tau}^2),\] induced by sending a generator $b\in \{X,Z,Y\}$ to the corresponding element $b_{\tau}$. The kernel of this homomorphism is generated by the element $T - \tau$. Under this morphism, $U_q^{\textrm{fin}}(-,+)$ is sent to $\textrm{Pol}(S_{q\tau}^2)$.\end{Prop}

\noindent \emph{Remarks}:\begin{itemize}\item  We note that also the standard Podle\'{s} sphere can be obtained in a similar manner, using instead the $^*$-algebra $U_q(0,+)$ from \cite{DeC2}.
\item From the observations in Section \ref{QUEA}, it follows that the action of $U_q(su(2))$ on the (localized) Podle\'{s} sphere is inner. This was also observed in \cite{Schm1}.
\item The isomorphisms $\sigma_{\tau}$ mentioned before the proposition are then easily seen to be induced from the automorphism $\sigma$ at the end of Section \ref{QUEA}.
\end{itemize}

\noindent The following result gives a classification of all irreducible $^*$-representations of $\textrm{Pol}(S_{q\tau}^2)$ (see \cite{Pod2}).

\begin{Prop}\label{PropReps} Any irreducible $^*$-representation of $\textrm{Pol}(S_{q\tau(x)}^2)$ on a Hilbert space is either \begin{itemize} \item faithful, in which case it is isomorphic to one of the following two $^*$-representations $\pi_{\pm}$ on $l^2(\mathbb{N})$: \[\left\{\begin{array}{lllllll}  Z_{\tau(x)} &\rightarrow & Z_{\tau(x),\pm}&\hspace{-0.3cm}:&e_k&\rightarrow& \pm q^{2k\mp x+1}e_k,\\ X_{\tau(x)}&\rightarrow & X_{\tau(x),\pm}&\hspace{-0.3cm}:&e_k &\rightarrow& \pm (1 - q^{2k})^{1/2}(1 + q^{2k \mp 2x})^{1/2}e_{k-1}.\end{array}\right.\]
\item one-dimensional, by sending $Z_{\tau}$ to zero and $X_{\tau}$ to a complex number of modulus 1.
\end{itemize}
\end{Prop}

\noindent Note that the above also classifies all irreducible representations of $U_q^{\textrm{fin}}(-,+)$, which were computed in \cite{Schm1}. It is obvious what is meant then by the $^*$-representation \[\pi_{\tau,\pm}: U_q^{\textrm{fin}}(-,+) \rightarrow B(l^2(\mathbb{N})).\] The equality \[\pi_{\tau,-} = \pi_{-\tau,+}\circ \sigma\] is easily observed. If we consider the pre-Hilbert space $V=\mathbb{C}\lbrack \mathbb{N}\rbrack$ with its natural orthonormal basis $e_k$, we can represent $\textrm{Pol}^{\textrm{ext}}(S_{q\tau}^{2})$ (and $U_q(-,+)$) as a $^*$-algebra of adjointable endomorphisms of $V$ (i.e.~ banded operators) by the same formulas as the one in the foregoing proposition. To avoid overloading the notation, we will make no distinction between an element in $\pi_{\pm}(\textrm{Pol}(S_{q\tau}^2))$ seen as an operator on $l^2(\mathbb{N})$ or its restriction to $V$ (this will not lead us astray).\\

\noindent Let us complete our notational conventions with the following.

\begin{Not}\label{NotCom} For $b\in \{X,Y,Z\}$, we will identify $b_{\tau}$ with the operator \[b_{\tau} := b_{\tau,-}\oplus b_{\tau,+} \qquad \in L(V\oplus V),\] and then write $\mathbf{b}_{\tau}$ for the operator \[\mathbf{b}_{\tau} := b_{-\tau,+}\oplus b_{\tau,+} \qquad \in L(V\oplus V).\]

\noindent We write \[\textrm{Pol}_+(S_{q\tau}^2),\textrm{Pol}_-(S_{q\tau}^2),\textrm{Pol}(S_{q\tau}^2),\textrm{Pol}_a(S_{q\tau}^2)\] for the images of $U_q^{\textrm{fin}}(-,+)$ under the respective representations, which we label by the same convention. \\
\end{Not}

\section{Equivariant Morita equivalence for the Podle\'{s} spheres}\label{SecEqMor}

\noindent In this section, we will prove Theorem \ref{Theo1}.\\

\noindent We will fix $x\in (-\infty,+\infty)$, and write $\tau:= \tau(x)$.\\

\begin{Not}\label{NotExp} For $w\in \{+,-,\;,a\}$, we consider $\textrm{Pol}_w^{\textrm{ext}}(S_{q\tau}^2)\odot M_2(\mathbb{C})$ with the right $U_q(su(2))$-module $^*$-algebra structure as at the end of Section \ref{SecQG}, using on $\mathbb{C}^2$ the spin $1/2$-representation $\pi_{1/2}$. We then denote by $\pi_{\tau,w}^{(2)}$ the morphism from $U_q(-,+)$ to $\textrm{Pol}_w^{\textrm{ext}}(S_{q\tau}^2)\odot M_2(\mathbb{C})$ as in Lemma \ref{LemMod}.2.\end{Not}

\noindent We will further denote by $\{e_+,e_-\}$ the canonical basis of $\mathbb{C}^2$, so that $U_q(su(2))$ acts by \[\left\{\begin{array}{lll}Ke_{\pm} &=& q^{\mp 1} e_{\pm}\\ E e_{\pm} &=& q^{1/2}\delta_{\pm,+}e_-\\Fe_{\pm} &=& q^{-1/2}\delta_{\pm,-}e_+\end{array}\right.\] We further denote the product basis elements $e_{k}\otimes e_{\pm}$ of $l^2(\mathbb{N})\otimes \mathbb{C}^2$ as $e_{k,\pm}$.

\begin{Prop}\label{PropEigv} The self-adjoint operator $T_{\tau,+}^{(2)}:=\pi_{\tau,+}^{(2)}(T)$ is bounded, its spectrum consisting of two eigenvalues. Moreover, it is an invariant element in $\textrm{Pol}_+(S_{q\tau}^2)\odot M_2(\mathbb{C})$.\end{Prop}

\begin{proof} First of all, it is clear that $T_{\tau,+}^{(2)}$ will be invariant, as $T$ is invariant for the adjoint action $\lhd$ (it is a central element of $U_q(-,+)$), and $\pi_{\tau,+}^{(2)}$ is equivariant. Then $T_{\tau,+}^{(2)}\in \textrm{Pol}_+(S_{q\tau}^2)\odot M_2(\mathbb{C})$ by Lemma \ref{LemFin} and the remark after it.\\

\noindent Next, a straightforward computation shows that $T_{\tau,+}^{(2)}$ preserves $\textrm{span}\{e_{k,+},e_{k+1,-}\}$ for $k\geq 0$, with the resulting 2-by-2-matrix being given by \[\left(\begin{array}{cc}((q^{-x}-q^x)q^{-1}-(q^{-1}-q)q^{2k-x+2})& \lambda^{-1}(1-q^{2k+2})^{1/2}(1+q^{2k-2x+2})^{1/2} \\ \lambda^{-1}(1-q^{2k+2})^{1/2}(1+q^{2k-2x+2})^{1/2} &  ((q^{-x}-q^x)q+(q^{-1}-q)q^{2k-x+2})\end{array}\right).\] (The remaining vector $e_{0,-}$ is an eigenvector, with eigenvalue the right lower corner of the above matrix with $k=-1$).\\

\noindent We find that the eigenvalues of these matrices are $q^{-(x+1)}-q^{x+1}$ and $q^{-(x-1)}-q^{x-1}$, and in particular are independent of $k$. This proves that $T_{\tau,+}^{(2)}$ has precisely two eigenvalues.\\

\end{proof}

\noindent \emph{Remark:} Note that the eigenvalues of $T_{\tau,+}^{(2)}$ naturally appear as \emph{differences} of $q$-powers, in contrast with the classical Casimir element of $U_q(su(2))$ whose eigenvalues are \emph{sums} of $q$-powers.\\

\noindent For further reference, we write down a basis of orthogonal eigenvectors for $T_{\tau,+}^{(2)}$.

\begin{Lem}\label{LemEig} An orthonormal set of eigenvectors for $T_{\tau,+}^{(2)}$ at eigenvalues $\tau(x\pm 1)$ is given by the $\xi^{\tau(x\pm1)}_{k,+}$ respectively, where $k\in \mathbb{N}$ and \begin{eqnarray*} \sqrt{1+q^{2x}}\cdot\xi^{\tau(x+1)}_{k,+} &=& (e_{k-1,+}\quad e_{k,-})\cdot\left(\begin{array}{ll} - (1-q^{2k})^{1/2}\\ q^{x}(1+q^{2k-2x})^{1/2}\end{array}\right),\\ \sqrt{1+q^{2x}}\cdot\xi^{\tau(x-1)}_{k,+} &=& (e_{k,+}\quad e_{k+1,-})\cdot \left(\begin{array}{l} q^x(1+q^{2k-2x+2})^{1/2} \\ (1-q^{2k+2})^{1/2}\end{array}\right).\end{eqnarray*}
\end{Lem}

\noindent We can also introduce an operator $T_{\tau,-}^{(2)}$ w.r.t.~ $\pi_{\tau,-}$ in a similar way, and the relation $\pi_{\tau,-} = \pi_{-\tau,+}\circ \sigma$ then immediately gives that $T_{\tau,-}^{(2)} = -T_{-\tau,+}^{(2)}$. We denote the respective eigenvectors for the eigenvalues $\tau(x\pm 1)$ of $T_{\tau,-}^{(2)}$ as \begin{eqnarray*}  \sqrt{1+q^{2x}}\cdot\xi^{\tau(x+1)}_{k,-} &=& (e_{k,+}\quad e_{k+1,-})\cdot \left(\begin{array}{l} (1+q^{2k+2x+2})^{1/2} \\ q^x(1-q^{2k+2})^{1/2}\end{array}\right),\\ \sqrt{1+q^{2x}}\cdot\xi^{\tau(x-1)}_{k,-} &=& (e_{k-1,+}\quad e_{k,-})\cdot\left(\begin{array}{ll} - q^x(1-q^{2k})^{1/2}\\ (1+q^{2k+2x})^{1/2}\end{array}\right).\end{eqnarray*}

\noindent We will also need to know the invariant functional on $\textrm{Pol}(S_{q\tau}^2)$. The following result was proven in \cite{Mim1} (see also \cite{Schm1}).

\begin{Prop}\label{PropFunc} Let $\varphi_{\tau}$ be the faithful normal positive functional on $B(l^2(\mathbb{N})\oplus l^2(\mathbb{N}))$ which has $\mathbf{Z}_{\tau}$ as its associated trace class operator. Then the restriction of $\varphi_{\tau}$ to $\textrm{Pol}(S_{q\tau}^2)$ is $U_q(su(2))$-invariant.\end{Prop}

\noindent One way to prove this is as follows: we want to show $\varphi_{\tau}(x\lhd b)=\varphi_{\tau}(x)\varepsilon(b)$ for $x\in \textrm{Pol}(S_{q\tau}^2)$ and $b\in \{X,Z,Y\}$. First show invariance for elements in $\textrm{Pol}(S_{q\tau}^2)\cdot Z$, which are trace class operators. One can use here the formulas in terms of the inner action without worrying about the unboundedness (of $Z^{-1}$ and the trace $Tr$). One is left with showing invariance for elements of the form $X^n$ or $(X^*)^n$ with $n\in \mathbb{N}$. But the only non-trivial case to consider is $n = 1$, for which we can simply compute the values.\\

\noindent We now give a proof of Theorem \ref{Theo1}.

\begin{proof}[Proof (of Theorem \ref{Theo1})]

\noindent Let us first note that we can apply Proposition \ref{PropEquiv} to $\textrm{Pol}(S_{q\tau}^2)\odot M_2(\mathbb{C})$, so that we can work on the level of $U_q(su(2))$.\\

\noindent We use the notation of Proposition \ref{PropEigv}. Write \[p\in \textrm{Pol}_+(S_{q\tau}^2)\odot M_2(\mathbb{C})\] for the eigenprojection of $T_{\tau,+}^{(2)}$ corresponding to the eigenvalue $\tau(x+1)$. Then by Proposition \ref{PropQuot}, the restriction of $\pi_{\tau,+}^{(2)}$ to $p(V\odot \mathbb{C}^2)$ factors through $\textrm{Pol}^{\textrm{ext}}(S_{q\tau(x+1)}^2)$. As the image of $Z_{\tau(x+1)}$ is easily seen to have distinct non-zero positive eigenvalues, it follows from the classification of $^*$-representations of the $\textrm{Pol}(S_{q\tau}^2)$ that this representation of $\textrm{Pol}(S_{q\tau(x+1)}^2)$ on $p(l^2(\mathbb{N})\otimes \mathbb{C}^2)$ is a copy of $\pi_{\tau(x+1),+}$. (In fact, one may check directly that the isomorphism is simply given by sending $\xi^{\tau(x+1)}_{k,+}$ to $e_k$.) The similar statements hold for the eigenspace of $\tau(x-1)$, as well as for the $\pi_{-}$-representations.\\

\noindent Let us denote $B = p(\textrm{Pol}(S_{q\tau}^2)\odot M_2(\mathbb{C}))p$. To see that \begin{equation}\tag{*}\label{Eqset} B = \textrm{Pol}(S_{q\tau(x+1)}^2), \end{equation} let us first remark that, by the preceding paragraph, the restriction of $\pi_{\tau(x)}^{(2)}$ to $p((V\oplus V)\otimes \mathbb{C}^2)$ is precisely $\pi_{\tau(x+1)}$. Hence $\textrm{Pol}(S_{q\tau(x+1)}^2) \subseteq B$ equivariantly. Further, if $\varphi_{1/2}$ is the invariant state on $M_2(\mathbb{C})$ for the adjoint spin 1/2-action of $SU_q(2)$, and $\varphi_{\tau}$ the invariant functional of the previous proposition, then $\varphi_{\tau}\otimes \varphi_{1/2}$ is invariant on $\textrm{Pol}(S_{q\tau}^2)\odot M_2(\mathbb{C})$. It follows that there exists a faithful normal functional on $B(l^2(\mathbb{N})\oplus l^2(\mathbb{N}))$ which restricts to an invariant functional on $B$. Now we remark that $B\subseteq B(l^2(\mathbb{N}))\oplus B(l^2(\mathbb{N})) \subseteq B(\mathscr{H}_{\tau(x+1)})$. As $B(l^2(\mathbb{N}))\oplus B(l^2(\mathbb{N}))$ is easily seen to equal $\textrm{Pol}(S_{q\tau(x+1)}^2)''$, we can conclude (\ref{Eqset}) by Lemma \ref{LemvN}.\\

\noindent Hence $S_{q\tau(x)}^2$ and $S_{q\tau(x+1)}^2$ are $SU_q(2)$-Morita equivalent by Proposition \ref{LemMor}. This proves that all Pod-le\'{s} spheres $S_{q\tau(x)}^2$ and $S_{q\tau(y)}$ with $x,y\in \mathbb{R}$ and $x-y\in \mathbb{Z}$ are $SU_q(2)$-Morita equivalent.\\

\noindent As the spin 1/2 representation is generating, it follows from Lemma \ref{LemInd} that $S_{q\tau(x)}$ is equivariantly Morita equivalent with some $\mathbb{X}$ iff $\mathbb{X}\cong S_{q\tau(y)}$ for some $y$ with $x-y\in \mathbb{Z}$. The statement of Theorem \ref{Theo1} now follows for $x\neq \infty$ by observing that $\textrm{Pol}(S_{q\tau(x)}) \cong \textrm{Pol}(S_{q\tau(-x)})$.\\

\noindent Finally, the standard Podle\'{s} sphere $S_{q0}^2$ is only equivariantly Morita equivalent with itself by the remark after Proposition \ref{LemMor}. Indeed, it is the quotient space of $SU_q(2)$ by $S^1$, but the latter only has one-dimensional irreducible representations, so any induced coaction is isomorphic to the original one.
\end{proof}

\noindent We end this section with the following observation.

\begin{Cor}[of the proof of Theorem \ref{Theo1}] With $\mathbb{G}=SU_q(2)$, the equivariant Picard group $\textrm{Pic}_{\mathbb{G}}(S_{q\tau}^2)$ of the Podle\'{s} spheres is determined as follows.
\begin{itemize} \item $\textrm{Pic}_{\mathbb{G}}(S_{q\infty}^2) \cong \mathbb{Z}$,
\item $\textrm{Pic}_{\mathbb{G}}(S_{q\tau(x)}^2) \cong \mathbb{Z}_2$ for $x\in \mathbb{Z}$,
\item $\textrm{Pic}_{\mathbb{G}}(S_{q\tau(x)}^2)$ is the trivial group in the remaining cases.
\end{itemize}
\end{Cor}

\noindent By the equivariant Picard group for a $\mathbb{G}$-homogeneous space $\mathbb{X}$, we mean the equivalence classes of equivariant equivalence $C(\mathbb{X})$-Hilbert bimodules, with composition given by the balanced $C(\mathbb{X})$-product.

\begin{proof} For $S_{q\infty}^2$ the result follows as the equivalence classes of irreducible imprimitivity Hilbert modules are labeled by $\mathbb{Z}= \textrm{Irrep}(S^1)$, and $\textrm{Pol}(S_{q\infty}^2)$ has no outer automorphisms. It is easily verified that the resulting group structure is also $\mathbb{Z}$. \\

\noindent For $S_{q\tau(x)}^2$ with $x\notin \mathbb{Z}$, we have computed that any irreducible imprimitivity Hilbert module has some $\textrm{Pol}(S_{q\tau(y)}^2)$ as its endomorphism algebra, where $y \in x+\mathbb{Z}$. As $S_{q\tau(x)} \cong S_{q\tau(y)}$ equivariantly iff $x = \pm y$, and as $S_{q\tau(x)}^2$ has no outer automorphisms, the result for this case also follows.\\

\noindent Finally, for $S_{q0}^2$, the first part of the previous argument still applies, but now $\textrm{Out}_{\mathbb{G}}(\textrm{Pol}(S_{q0}^2)) = \mathbb{Z}_2$. Hence $\textrm{Pic}_{\mathbb{G}}(S_{q0}^2) \cong \mathbb{Z}_2$. As the $S_{q\tau(x)}$ with $x\in \mathbb{Z}$ are $\mathbb{G}$-Morita equivalent with $S_{q0}^2$, the result follows also for these cases. (In fact, observe that the $\textrm{Pol}(S_{q\tau(2l)}^2)$-linear span of the $A_s$ inside the $^*$-algebra $B_l$ of Theorem \ref{Theo2} give a concrete equivariant equivalence (pre-)Hilbert C$^*$-bimodule between $\textrm{Pol}(S_{q\tau(2l)})$ and $\textrm{Pol}(S_{q\tau(-2l)})\cong \textrm{Pol}(S_{q\tau(2l)})$.)
\end{proof}

\noindent \emph{Remark:} A similar distinction between the equilateral Podle\'{s} spheres (i.e.~ $\tau\neq 0$ or $\infty$), and the only further one that we are aware of, appears in Proposition 9 of \cite{Hec1}, where the spectral decomposition of a certain subspace of the restricted duals of the Podle\'{s} spheres is computed. However, now the set of exceptional cases is slightly larger, as the are given by the $\tau(x)$ with $x\in \frac{1}{2}\mathbb{N}$. We have not examined in detail whether there is any direct connection with the above result.

\section{Equivariant Morita equivalence for the quantum projective plane}

\noindent We will first show that the module $^*$-algebra $B_l$ of Theorem \ref{Theo2} is well-defined.\\

\noindent We will index the elements $b\in \{X,Y,Z\}\subseteq B_{l}$ with $2l$ (dropping the $\tau$-symbol w.r.t. ~ previous notation), and denote the $A_s$ as $A_{2l}^{(s)}$. However, when the indices are not crucial in a computation, we will drop them.\\

\noindent We will also use the following orthogonal basis for the pre-Hilbert space $V\oplus V$, where $V = \mathbb{C}\lbrack \mathbb{N}\rbrack$: for $k \in \mathbb{N}-2l = \{-2l,-2l+1,\ldots\}$, we denote by $e_{k,+}^{(l)}$ the vector $0\oplus e_{k+2l}$, and for $k\in \mathbb{N}$, we denote by $e_{k,-}^{(l)}$ the vector $e_{k}\oplus 0$.\\

\noindent We want to build now on $V\oplus V$ a bounded $^*$-representation of $B_{l}$. Namely, we let the generators of $B_l$ correspond to the following banded operators: \begin{eqnarray*} X_{2l} e_{k,\pm}^{(l)} &=& \pm(1 \pm q^{2k})^{1/2}(1 \mp q^{2k+4l})^{1/2}e_{k-1,\pm}^{(l)},\\ Z_{2l} e_{k,\pm}^{(l)} &=& \pm q^{2k+2l+1} e_{k,\pm}^{(l)}\\ X_{2l}^*e_{k,\pm}^{(l)} &=& \pm (1\pm q^{2k+2})^{1/2}(1\mp q^{2k+4l+2})^{1/2} e_{k+1,\pm}^{(l)} \\ A_{2l}^{(s)} e_{k,\pm}^{(l)} &=& (\pm 1)^s (\pm q^{2k+2s+2};q^2)_{2l-s}^{1/2} (\mp q^{2k+2};q^2)_{2l+s}^{1/2} e_{k+s,\mp}^{(l)}\end{eqnarray*} It is an easy task to check that the commutation relations in Theorem \ref{Theo2} are satisfied for these operators. If we restrict to $\textrm{Pol}(S_{q\tau(2l)}^2) \subseteq B_l$, we see that we get the natural representation $\pi_{\tau(2l)}$.\\

\begin{Lem}\label{LemFaith} The above representation is faithful.\end{Lem}

\begin{proof}
\noindent Let us formally write $X^{-1}$ for $X^*$. Then the commutation relations, together with their adjoints, clearly allow to write any element of $B_l$ as a linear combination of elements of the form \begin{itemize} \item $X^mZ^n$ with $m\in \mathbb{Z}, n\in \mathbb{N}$,\item $A_sZ^n$ with $s\in \{-2l+1,-2l+1,\ldots, 2l-1\}$ and $n\in \mathbb{N}$, \item $A_{-2l}X^mZ^n$ with $m,n\in \mathbb{N}$, and \item $A_{2l}(X^*)^{m}Z^n$ with $m,n\in \mathbb{N}$.\end{itemize} We will now show that the representations of these monomials are linearly independent. We will in the following already use the same notation for these operator algebraic implementations. Note that in any case none of the above monomials are zero operators.\\

\noindent From looking at the natural $\mathbb{Z}\times \mathbb{Z}_2$-gradation on $B_l$ by the adjoint action of $Z$, we see immediately that the above families are linearly independent amongst each other, and that inside each family we can only have linear dependencies of the form $A_{s}X^m P(Z)=0$ for some non-zero (Laurent) polynomial $P$ in $Z$. But it is clear that these do not occur.

\end{proof}

\noindent Using the notation from Notation \ref{NotCom}, consider $Pol_a^{\textrm{ext}}(S_{q\tau(2l)}^2) \subseteq L(V\oplus V)$, which induces a module $^*$-algebra structure on the $^*$-algebra of banded operators on $V\oplus V$ by Lemma \ref{LemMod}.1.

\begin{Prop} The above module $^*$-algebra structure restricts to $B_l$, and coincides with the one described in Theorem \ref{Theo2}.\end{Prop}

\begin{proof} It is clear that the above module $^*$-algebra structure restricts to $\textrm{Pol}(S_{q\tau(2l)}^2)\subseteq B_l$, and coincides with the usual one. We are therefore left to show that it behaves in the right way on the operators $A_s$.\\

\noindent Let us write $e$ for the sign operator \[e:V\oplus V\rightarrow V\oplus V: e_{k,\pm}^{(l)} \rightarrow \pm e_{k,\pm}^{(l)},\] so that $\mathbf{b}_{2l} = eb_{2l}$ for $b\in \{X,Z,Y\}$. Denote $\theta_{2l}^{(s)} = \lambda_s A_{s}$ where \[\lambda_s=q^{\frac{1}{2}s(s-1)} \frac{(q^{4l-2s+2};q^2)_{s+2l}^{1/2}}{(q^{2};q^2)_{s+2l}^{1/2}}.\] We then have \begin{eqnarray*} \theta_{2l}^{(s)} \lhd K &=& \mathbf{Z} \theta_{2l}^{(s)} \mathbf{Z}^{-1}\\ &=& eZ \theta_{2l}^{(s)} Z^{-1}e \\ &=& - q^{2s} e \theta_{2l}^{(s)} e \\ &=& q^{2s} \theta_{2l}^{(s)}.\end{eqnarray*} This small computation already makes it clear why we can not use the naive map of $U_q(-,+)$ into $\textrm{Pol}(S_{q\tau(2l)}^2)$ to define the module $^*$-algebra structure.\\

\noindent Similarly, we compute \begin{eqnarray*} \theta_{2l}^{(-2l)} \lhd (q^{-1/2}\lambda^{-1}E) &=& Z^{-1}e\lbrack \theta_{2l}^{(-2l)},eX\rbrack \\ &=& Z^{-1}e(\theta_{2l}^{(-2l)}Xe-eX\theta_{2l}^{(-2l)}) \\ &=& -Z^{-1}(\theta_{2l}^{(-2l)} X + X\theta_{2l}^{(-2l)}) \\ &=& -Z^{-1}(\theta_{2l}^{(-2l)}X-\theta_{2l}^{(-2l)}X) \\ &=& 0,\end{eqnarray*} showing that $\theta_{2l}^{(-2l)}$ is a highest weight vector for the spin $2l$-representation.\\

\noindent For $s>-2l$, we have \[ X\theta_{2l}^{(s)} = -\frac{\lambda_s}{\lambda_{s-1}}(1-q^{2l+1}Z) \theta_{2l}^{(s-1)},\] and, by taking the adjoint of the commutation relations for the $X^*$, we also have \[\theta_{2l}^{(s)}X = \frac{\lambda_s}{\lambda_{s-1}}(1-q^{-2s-2l+1}Z)\theta_{2l}^{(s-1)}.\] So then we find \[\theta_{2l}^{(s)} \lhd (q^{-1/2}\lambda^{-1}E)  = \frac{\lambda_s}{\lambda_{s-1}} (q^{-2s-2l+1}-q^{2l+1})\theta_{2l}^{(s-1)}.\] Simplifying, this becomes \[\theta_{2l}^{(s)} \lhd E = q^{-s-2l+\frac{1}{2}}\lambda(1-q^{4l+2s})^{1/2}(1-q^{4l-2s+2})^{1/2} \theta_{2l}^{(s-1)}.\] Carrying out a similar calculation for $F$, or using the compatibility between the module structure and the $^*$-operation, we also find \[\theta_{2l}^{(s)} \lhd F = q^{s-2l-\frac{1}{2}}\lambda(1-q^{4l+2s+2})^{1/2}(1-q^{4l-2s})^{1/2} \theta_{2l}^{(s+1)}\] for $s<2l$, and $\theta_{2l}^{(2l)} \lhd F= 0$. In all, we find that the action of $U_q(su(2))$ on the elements $\theta_{2l}^{(s)}$ indeed gives a (right) presentation of the spin $2l$-representation.
\end{proof}

\noindent We now show that the action on $B_l$ is ergodic.

\begin{Lem}\label{LemInvs} The module $^*$-algebra $B_l$ has only the scalar multiples of the unit as its invariant elements.\end{Lem}

\begin{proof} Using the arguments concerning the basis constructed in Lemma \ref{LemFaith}, we see that an invariant element $b$ can be written as $b_1+b_2$ with $b_1 \in \textrm{Pol}(S_{q\tau(2l)}^2)$ and $b_2$ a linear combination of elements of the form $A_{0}Z^n$ with $n\in \mathbb{N}$. As the natural grading on $B_l$ is $U_q(su(2))$-compatible, both $b_1$ and $b_2$ have to be invariant. But the action on a Podle\'{s} sphere is ergodic, so $b_1$ reduces to a scalar. On the other hand, set $b_2 = A_0 P(Z)$ with $P(Z)$ a polynomial in $Z$. Then the invariance of $b_2$ under the adjoint action of $E$ leads to the following functional equation for $P$: \[(1-q^{-2l-1}Z)P(-q^{-2}Z) = (1+q^{2l-1}Z)P(Z).\] It is clear that the only solution is $P=0$.
\end{proof}

\noindent We can thus apply Proposition \ref{PropEquiv} to find that $B_l$ has a well-defined action by $SU_q(2)$. This finishes the existence part of the $SU_q(2)$-action proposed in Theorem \ref{Theo2}.\\

\noindent It is also easy to provide the invariant functional on $B_l$.\\

\begin{Prop} Let $\varphi_{l}$ be the normal positive functional on $B(l^2(\mathbb{N})\oplus l^2(\mathbb{N}))$ which has $\mathbf{Z}$ as its associated trace class operator. Then $\varphi_l$ is $U_q(su(2))$-invariant on $B_l$.\end{Prop}

\begin{proof} Let $p_{\pm}$ be the projections onto the summands of $l^2(\mathbb{N})\oplus l^2(\mathbb{N})$. We find that the conditional expectation \[ E: B(l^2(\mathbb{N})\oplus l^2(\mathbb{N})) \rightarrow B(l^2(\mathbb{N}))\oplus B(l^2(\mathbb{N})): x\rightarrow p_-xp_- + p_+xp_+\] restricts to an equivariant conditional expectation $B_l \rightarrow \textrm{Pol}(S_{q\tau(2l)}^2)$. Since $\varphi_l = \varphi_{\tau(2l)}\circ E$, the proposition follows from Proposition \ref{PropFunc}.
\end{proof}

\noindent We can now prove Theorem \ref{Theo2}.

\begin{proof}[Proof (of Theorem \ref{Theo2})]

\noindent We first remark that the definition of $B_l$ also makes sense when $l=0$. In fact, it is easily seen that $B_0$ is just a copy of $\textrm{Pol}(S_{q,0}^2)\rtimes \mathbb{Z}_2$ where $\mathbb{Z}_2$ acts by the automorphism $\sigma_0$ (see the remark after Definition \ref{DefPod}). All results of this section then hold for $B_0$, \emph{except} that $B_0$ is not ergodic: the proof of \ref{LemInvs} in fact shows that the space of invariants is linearly spanned by 1 and $A_0$. Now the `antipodal reflection map' $\sigma_0$ on $\textrm{Pol}(S_{q0}^2)$ gives a Galois action by $\mathbb{Z}_2$ (cf. \cite{Haj2}, Proposition 2.10). Hence, by the discussion after Proposition \ref{PropMor2}, $B_0$ is $SU_q(2)$-equivariantly Morita equivalent with $\textrm{Pol}(\mathbb{R}P_q^2)$, which is by definition the fixed point algebra under $\sigma_0$. If we denote $p_{\pm} =\frac{1}{2}(1\pm A_0)$, then $p_{\pm}B_0p_{\pm} \cong \textrm{Pol}(\mathbb{R}P_q^2)$ equivariantly. \\

\noindent Now for $l\in \frac{1}{2}\mathbb{N}$, let us write $V_{2l,\pm}$ for the space $V=\mathbb{C}\lbrack \mathbb{N}\rbrack$ considered with the $\pi_{\pm}$-action by $\textrm{Pol}(S_{q\tau(2l)}^2)$, and $V_{2l} = V_{2l,-}\oplus V_{2l,+}$. Consider $B_l\odot M_2(\mathbb{C})$, represented on $V_{2l}\odot \mathbb{C}^2$. Let us write the eigenvectors $\xi$ from Lemma \ref{LemEig} and the remark under it as follows: \[ \left.\begin{array}{llll} e^{(l\pm\frac{1}{2})}_{k,+}&=& \xi^{\tau(2l\pm 1)}_{k+2l\pm 1,+}, &\qquad k \in \mathbb{N}-(2l\pm 1),\\
e^{(l\pm \frac{1}{2})}_{k,-}&=& \xi^{\tau(2l\pm 1)}_{k,-}, &\qquad k\in \mathbb{N}.\end{array}\right.\]

\noindent We may identify the span of the $e^{(l \pm \frac{1}{2})}_{k,\nu}$ over all $k$ with $V_{2l\pm 1,\nu}$, where $\nu\in \{-,+\}$. Then we can write $V_{2l}\odot \mathbb{C}^2$ as $V_{2l-1} \oplus V_{2l+1}$, with corresponding projection maps $p_{2l \pm 1}$. From the results of the previous section, it follows that \[p_{2l \pm 1} (\textrm{Pol}(S_{q\tau(2l)}^2)\otimes M_2(\mathbb{C}))p_{2l \pm 1} = \textrm{Pol}(S_{q\tau(2l\pm 1)}^2),\] in its natural presentation w.r.t. the basis $e^{(l\pm \frac{1}{2})}$. Now as $\pi_{-\tau,+} = \pi_{\tau,-}\circ \sigma$, we have that in the new basis also \[\pi_{\tau(2l),a}^{(2)} = \pi_{\tau(2l-1),a}\oplus \pi_{\tau(2l+1),a},\] where we recall the notations Notation \ref{NotCom} and Notation \ref{NotExp}. By Lemma \ref{LemMod}.2, the action of $U_q(su(2))$ on $B_l\odot M_2(\mathbb{C})$ will be implemented by this representation.\\

\noindent We want to show now that \[p_{2l \pm 1} (B_{l}\otimes M_2(\mathbb{C}))p_{2l \pm 1} = B_{l \pm \frac{1}{2}},\] where for the moment we assume $l>0$ in the $-$-case. As the $\sigma$-weak closure of $B_{l}$ is clearly the whole of $B(l^2(\mathbb{N})\oplus l^2(\mathbb{N}))$, and as the latter has a normal positive functional which restricts to an invariant functional on $B_l$, by the previous proposition, a similar argument as in the proof of Theorem \ref{Theo1} shows that it is sufficient to prove that the right hand side is \emph{contained} in the left hand side.\\

\noindent We have already remarked above that the copy of the Podle\'{s} sphere inside $B_{l\pm \frac{1}{2}}$ will belong to the left hand side. It remains to prove this also for the generators $A_{2l\pm 1}^{(s)}$.\\

\noindent Let us denote $e_{k,\mu,\nu}$ for the vector $e_{k,\mu}^{(l)}\otimes e_{\nu}$ in $V_{2l}\otimes \mathbb{C}^2$. Then we may write \begin{eqnarray*} \sqrt{1+q^{4l}}\cdot e^{(l+\frac{1}{2})}_{k,\pm} &=& (e_{k,\pm,+}\quad e_{k+1,\pm,-})\cdot\left(\begin{array}{ll} \mp (1\mp q^{2k+4l+2})^{1/2}\\ q^{2l}(1\pm q^{2k+2})^{1/2}\end{array}\right),\\
\sqrt{1+q^{4l}}\cdot e^{(l-\frac{1}{2})}_{k,\pm} &=& (e_{k-1,\pm,+}\quad e_{k,\pm,-})\cdot\left(\begin{array}{ll} \pm q^{2l}(1 \pm q^{2k})^{1/2}\\ (1 \mp q^{2k+4l})^{1/2}\end{array}\right).\end{eqnarray*}

\noindent Inversely we have \begin{eqnarray*} \frac{1}{\sqrt{1+q^{4l}}}\cdot e_{k,\pm,+} &=& (e_{k,\pm}^{(l+\frac{1}{2})}\quad e_{k+1,\pm}^{(l-\frac{1}{2})})\cdot\left(\begin{array}{ll} \mp (1\mp q^{2k+4l+2})^{1/2}\\ \pm q^{2l}(1\pm q^{2k+2})^{1/2}\end{array}\right),\\
\frac{1}{\sqrt{1+q^{4l}}}\cdot e_{k,\pm,-} &=& (e_{k-1,\pm}^{(l+\frac{1}{2})}\quad e_{k,\pm}^{(l-\frac{1}{2})})\cdot\left(\begin{array}{ll} q^{2l}(1 \pm q^{2k})^{1/2}\\ (1 \mp q^{2k+4l})^{1/2}\end{array}\right).\end{eqnarray*}

\noindent One computes then that w.r.t.~ the original basis of $V_{2l}\otimes \mathbb{C}^2 = \left(\begin{array}{ll} V_{2l}\otimes e_+ \\ V_{2l}\otimes e_- \end{array}\right)$, one has \[ A_{2l+1}^{(0)} = \left(\begin{array}{cc} - A_{2l}^{(0)} (1-q^{4l+2}Z^2) & q^{2l}A_{2l}^{(-1)}(1+q^{-2l-1}Z)(1+q^{2l+1}Z) \\ -q^{2l}A_{2l}^{(1)}(1-q^{-2l+1}Z)(1-q^{2l+1}Z) & q^{4l}A_{2l}^{(0)}(1-q^{-4l-2}Z^2) \end{array}\right).\]

\noindent A similar computation shows that, for $l>0$, we can write \begin{equation}\tag{*} \label{IdEq} A_{2l-1}^{(0)} = \left(\begin{array}{cc} - q^{4l} A_{2l}^{(0)} & -q^{2l} A_{2l}^{(-1)} \\ q^{2l} A_{2l}^{(1)}& A_{2l}^{(0)}\end{array}\right).\end{equation}

\noindent We have thus shown that \[A_{2l\pm 1}^{(0)} \in p_{2l \pm 1} (B_{l}\otimes M_2(\mathbb{C}))p_{2l \pm 1}.\] As all other $A^{(s)}_{2l\pm1}$ lie in $A_{2l\pm1}^{(0)} \lhd U_q(su(2))$, it follows that \[A_{2l\pm 1}^{(s)} \in p_{2l \pm 1} (B_{l}\otimes M_2(\mathbb{C}))p_{2l \pm 1}\] for all $s$, and so $p_{2l \pm 1} (B_{l}\otimes M_2(\mathbb{C}))p_{2l \pm 1} = B_{l \pm \frac{1}{2}}$.\\

\noindent Now Theorem 2 will follow from Lemma \ref{LemInd} and the above discussion, if we can also show that $\textrm{Pol}(\mathbb{R}P_q^2) \otimes M_2(\mathbb{C}) \cong B_{\frac{1}{2}}$. From the remarks in the first paragraph of this proof, it is sufficient to show that \[(p_+\otimes 1)(B_0\otimes M_2(\mathbb{C}))(p_+\otimes 1) \cong B_{\frac{1}{2}}.\] Now an easy computation shows that $(A_0\otimes 1)p_1(A_0\otimes 1) = p_{-1}$. As we already know that \[p_1(B_0\otimes M_2(\mathbb{C}))p_1 \cong B_\frac{1}{2},\] on which the $SU_q(2)$-action is ergodic, we must have that $p_1$ and $p_{-1}$ are minimal projections in the fixed point algebra of $B_0\otimes M_2(\mathbb{C})$. As also $(A_0\otimes 1)$ lies in the latter, it follows that the fixed point algebra is in fact $M_2(\mathbb{C})$. Hence \[(p_+\otimes 1)(B_0\otimes M_2(\mathbb{C}))(p_+\otimes 1)\cong p_1(B_0\otimes M_2(\mathbb{C}))p_1 \cong B_\frac{1}{2},\] and we are done.

\end{proof}

\noindent \emph{Acknowledgements:} I would like to thank U. Kr\"{a}hmer for pointing me towards the references \cite{Hec1} and \cite{Schm1}.

\end{document}